\documentclass[preprint,12pt]{elsarticle}
\usepackage{amsmath,amssymb,amsthm}
\usepackage{graphicx}
\usepackage{color}

\makeatletter
\def\ps@pprintTitle{%
  \let\@oddhead\@empty
  \let\@evenhead\@empty
  \def\@oddfoot{\reset@font\hfil\thepage\hfil}
  \let\@evenfoot\@oddfoot}
\makeatother

\newtheorem{lettertheorem}{Theorem}

\newtheorem{defin}{Definition}[section]

\newtheorem{theorem}[defin]{Theorem}
\newtheorem{lemma}[defin]{Lemma}

\newtheorem{qu}{Question}
\newtheorem{exa}[defin]{Example}

\begin{document}

\begin{frontmatter}

\title{Three results on transcendental meromorphic solutions of certain  nonlinear differential equations}

\author[mymainaddress]{Nan Li\corref{cor1}} 
\ead{nanli32787310@163.com}


\author[mysecondaryaddress]{Lianzhong Yang}
\ead{lzyang@sdu.edu.cn}

\cortext[cor1]{Corresponding author: Nan Li}

\address[mymainaddress]{University of Jinan, School of Mathematical Sciences, Jinan, Shandong, 250022, P.R. China}

\address[mysecondaryaddress]{School of Mathematics, Shandong University,
  Jinan, Shandong Province, 250100, P.R.China}

\begin{abstract}
In this paper, we study the  transcendental meromorphic  solutions
  for the nonlinear differential equations:  $f^{n}+P(f)=R(z)e^{\alpha(z)}$
  and $f^{n}+P_{*}(f)=p_{1}(z)e^{\alpha_{1}(z)}+p_{2}(z)e^{\alpha_{2}(z)}$
  in the complex plane, where $P(f)$ and $P_{*}(f)$ are differential polynomials in $f$ of
  degree $n-1$ with coefficients
  being small functions and rational functions respectively,  $R$ is a non-vanishing small function of $f$,  $\alpha$ is a nonconstant entire function,  $p_{1}, p_{2}$ are non-vanishing rational functions, and $\alpha_{1}, \alpha_{2}$  are nonconstant polynomials. Particularly,  we consider the  solutions
 of the second equation when $p_{1}, p_{2}$ are nonzero constants, and
 $\deg \alpha_{1}=\deg \alpha_{2}=1$. Our results are improvements and complements of  Liao (Complex Var. Elliptic Equ. 2015, 60(6): 748--756),
 and Rong-Xu (Mathematics 2019, 7, 539), etc., which partially answer a question proposed by Li (J. Math. Anal. Appl. 2011, 375: 310--319).
\end{abstract}

\begin{keyword}
Meromorphic functions \sep nonlinear differential equations \sep
small functions \sep differential polynomials.

\MSC[2010] 34M05 \sep 30D30 \sep 30D35
\end{keyword}
\end{frontmatter}

\section{Introduction}

Let $f(z)$ be a transcendental meromorphic function in the complex plane
$\mathbb{C}$.  We assume that the reader is familiar with the standard notations and main results in Nevanlinna theory (see \cite{Hayman},\cite{Laine},\cite{Yi1}).
Throughout this paper, the term $S(r,f)$ always has the property that
$S(r,f)=o(T(r,f))$ as $r\to \infty$, possibly outside a set $E$ (which is not necessarily the same at each occurrence) of finite
linear measure. A meromorphic function $a(z)$ is said to be a small function
with respect to $f(z)$ if and only if $T(r,a)=S(r,f)$. In addition,
$N_{1)}(r,1/f)$ and $N_{(2}(r,1/f)$  are used to denote the counting functions corresponding to simple  and multiple zeros of $f$, respectively.

In the past few decades,  many scholars, see [7-10] etc., focus on the solutions of the nonlinear differential equations of the form
\begin{equation}\label{introeq1}
  f^{n}+P(f) =h,
\end{equation}
where $P(f)$ denotes a differential polynomial in $f$
 of degree at most
 $n-2$, and
$h$ is a given  meromorphic function.

In 2015, Liao  \cite{LIAO2015} investigated the forms of
meromorphic solutions of equation \eqref{introeq1} for specific
$h$, and obtained the following result.
\begin{lettertheorem}\label{thmAB}
  Let $n\geq 2$ and $P(f)$ be a differential polynomial in $f$ of
  degree $d$ with rational functions as its coefficients. Suppose that
  $p$ is a non-zero rational function, $\alpha$ is a non-constant polynomial
  and $d\leq n-2$. If the following differential equation
   \begin{eqnarray}\label{introeq1ddd}
     f^{n}+P(f)=p(z)e^{\alpha(z)},
   \end{eqnarray}
 admits a meromorphic function $f$ with finitely many poles, then
  $f$ has the following form $f(z)=q(z)e^{r(z)}$ and $P(f)\equiv 0$,
  where $q(z)$ is a rational function and $r(z)$ is a polynomial
  with $q^{n}=p, nr(z)=\alpha(z)$. In particular, if $p$ is a
  polynomial, then $q$ is a polynomial, too.
  \end{lettertheorem}

 If the condition $d\leq n-2$  is omitted, then the  conclusions in Theorem~\ref{thmAB} can not hold. For example,
 $f_{0}(z)=e^{z}-1$ is a solution of the equation $f^{2}+f'+f=e^{2z}$, here
$n=2$ and $d=1=n-1$. So it's natural to ask what will happen
to the solutions of  equation \eqref{introeq1ddd} when
$d=n-1$? In this paper, we study this problem and obtain the
 following result, which is a complement of Theorem~\ref{thmAB}.

\begin{theorem}\label{thmABABnew1}
 Let $n\geq 2$ be an integer and $P(f)$ be a differential polynomial in $f$ of
  degree $n-1$ with coefficients
  being small functions. Then for any
 entire function $\alpha$ and any small
  funtion $R$, if the equation
  \begin{equation}\label{thmABABneweq1}
    f^{n}+P(f)=R(z)e^{\alpha(z)}
  \end{equation}
 possesses a meromorphic solution $f$
  with $N(r,f)=S(r,f)$, then
$f$ has the following form:
 \begin{eqnarray*}
   f(z)=s(z)e^{\alpha(z)/n}+\gamma(z),
 \end{eqnarray*}
where $s$ and $\gamma$ are small functions of $f$ with $s^{n}=R$.
\end{theorem}

The following Example~\ref{exaz1}
  shows that the case  in Theorem~\ref{thmABABnew1} occurs.

\begin{exa}\label{exaz1}
  $f_{0}=e^{z}+1$ is a solution of the following equation
 \begin{eqnarray*}
   f^{3}-2ff'-(f')^{2}-f=e^{3z}.
\end{eqnarray*}
Here, $P(f)=-2ff'-(f')^{2}-f$,  $n=3$, and $\deg P(f)=2=n-1$.
\end{exa}

In 2011, Li \cite{LiP2011} considered to find all entire solutions of equation
\eqref{introeq1} for $h=p_{1}e^{\alpha_{1} z}+p_{2}e^{\alpha_{2} z}$, where $\alpha_{1}$ and $\alpha_{2}$ are distinct constants, and obtained the following result.

\begin{lettertheorem}\label{thmcccc}
Let $n\geq 2$ be an integer, $P(f)$ be a differential polynomial in $f$ of
  degree at most $n-2$ and $\alpha_{1},\, \alpha_{2},\, p_{1},\, p_{2}$ be nonzero constants satisfying $\alpha_{1}\neq\alpha_{2}$. If $f$ is a transcendental meromorphic solution of the following equation
  \begin{eqnarray}\label{thmcccceq2}
    f^{n}(z)+P(f)=p_{1}e^{\alpha_{1} z}+p_{2}e^{\alpha_{2} z}
  \end{eqnarray}
 satisfying $N(r,f)=S(r,f)$, then one of the following relations holds:
\begin{itemize}
 \item [(1)].  $f= c_{0}+c_{1}e^{\frac{\alpha_{1} z}{n}}$;
 \item [(2)]. $f= c_{0}+c_{2}e^{\frac{\alpha_{2} z}{n}}$;
  \item [(3)].  $f=c_{1}e^{\frac{\alpha_{1} z}{n}}+c_{2}e^{\frac{\alpha_{2} z}{n}}$ and $\alpha_{1}+\alpha_{2}=0$.
 \end{itemize}
where $c_{0}(z)$ is a small function of $f$ and constants $c_{1}$
and $c_{2}$ satisfy $c_{1}^{n}=p_{1}$ and $c_{2}^{n}=p_{2}$, respectively.
\end{lettertheorem}

For further study, Li \cite{LiP2011} proposed the following question:

\begin{qu}\label{quest1}
How to find the solutions of equation \eqref{thmcccceq2} under the condition $\deg P(f)=n-1$ ?
\end{qu}

For the case $\alpha_{2}=-\alpha_{1}$, Li \cite{LiP2011} has already given the detailed forms of the entire solutions of equation \eqref{thmcccceq2} when $\deg P(f)=n-1$; For the case $\alpha_{2}=\alpha_{1}$, \eqref{thmcccceq2} can be reduced to $f^{n}+P(f)=(p_{1}+p_{2})e^{\alpha_{1} z}$, then we can get the forms of entire solutions by using Theorem~\ref{thmABABnew1}.
 So it's natural to ask: what will happen when $\alpha_{2}\pm\alpha_{1}\neq 0$.

Chen and Gao \cite{chen2018JCAA} studied the above question,  and obtained the following result.

\begin{lettertheorem}\label{thmddd}
Let $a(z)$ be a nonzero polynomial and $p_{1},\, p_{2},\, \alpha_{1},\, \alpha_{2}$ be nonzero constants such that $\alpha_{1}\neq \alpha_{2}$. Suppose that $f(z)$ is a transcendental entire solution of finite order of  the differential equation
 \begin{eqnarray}\label{thmcccceq1}
    f^{2}(z)+a(z)f'(z)=p_{1}e^{\alpha_{1} z}+p_{2}e^{\alpha_{2} z}
  \end{eqnarray}
satisfying $N(r,1/f)=S(r,f)$, then $a(z)$ must be a constant and one of
the following relations holds:
\begin{itemize}
 \item [(1)].  $f= c_{1}e^{\frac{\alpha_{1} z}{2}}$, $ac_{1}\alpha_{1}=2p_{2}$
 and $\alpha_{1}=2\alpha_{2}$;
 \item [(2)]. $f= c_{2}e^{\frac{\alpha_{2} z}{2}}$, $ac_{2}\alpha_{2}=2p_{1}$
 and $\alpha_{2}=2\alpha_{1}$,
 \end{itemize}
where  $c_{1}$ and $c_{2}$ are constants
satisfying $c_{1}^{2}=p_{1}$ and $c_{2}^{2}=p_{2}$, respectively.
\end{lettertheorem}

Later, Rong and Xu
\cite{RongXu2019} improved Theorem~\ref{thmddd} by removing the condition that $f(z)$ is a finite-order function.  In \cite{RongXu2019}, they
also considered the general case in Question~\ref{quest1}, and
 obtained the following result.
\begin{lettertheorem}\label{thmeee}
Let $n\geq 2$ be an integer. Suppose that $P(f)$ is a differential polynomial in $f(z)$ of degree $n-1$ and that $\alpha_{1},\, \alpha_{2},\, p_{1}$ and $p_{2}$ are nonzero constants such that $\alpha_{1}\neq\alpha_{2}$. If $f(z)$ is a transcendental meromorphic solution of the differential equation \eqref{thmcccceq2}
satisfying $N(r,f)=S(r,f)$, then $\rho(f)=1$ and one of the following relations holds:
\begin{itemize}
 \item [(1)].  $f(z)=c_{1}e^{\frac{\alpha_{1} z}{n}}$ and $c_{1}^{n}=p_{1}$;
 \item [(2)]. $f(z)= c_{2}e^{\frac{\alpha_{2} z}{n}}$ and $c_{2}^{n}=p_{2}$, where $c_{1}$ and $c_{2}$ are constants;
  \item [(3)].  $T(r,f)\leq N_{1)}(r,1/f)+T(r,\varphi)+S(r,f)$,
  where  $\varphi\, (\not\equiv 0)$ is
  equal to $\alpha_{1}\alpha_{2}f^{2}-n(\alpha_{1}+\alpha_{2})ff'
  +n(n-1)(f')^{2}+nff''$.
 \end{itemize}
\end{lettertheorem}

In this paper, we go on investigating Question~\ref{quest1} and obtain the following results, which are improvements of Theorems~\ref{thmddd} and~\ref{thmeee}.

\begin{theorem}\label{thmA1z1}
Let $n\geq 2$ be an integer. Suppose that $P_{*}(f)$ is a differential polynomial in $f(z)$ of degree $n-1$ and with rational functions as its coefficients, $\alpha_{1}, \alpha_{2}$ be nonconstant polynomials, and
 $p_{1}, p_{2}$ be non-vanishing rational functions. If $f(z)$ is a transcendental
  meromorphic solution
 of the following nonlinear differential equation
 \begin{equation}\label{aaaa1}
f^{n}(z)+P_{*}(f)=p_{1}(z)e^{\alpha_{1}(z)}+p_{2}(z)e^{\alpha_{2}(z)},
\end{equation}
 with $\lambda_{f}=\max\{\lambda(f),\lambda(1/f)\}<\sigma(f)$, then $\sigma(f)=\deg \alpha_{1}=\deg \alpha_{2}$, and one of the following relations holds:
\begin{itemize}
 \item [(I)]. $\alpha_{2}'=\alpha_{1}'$. In this case, $f=s_{1}(z)\exp(\alpha_{1}(z) /n)=s_{2}(z)\exp(\alpha_{2}(z) /n)$, where $s_{1}$ and $s_{2}$ are rational functions satisfying
     $s_{1}^{n}=p_{1}+p_{2}c_{2}$ and $s_{2}^{n}=\frac{1}{c_{2}}p_{1}+p_{2}$,
     $c_{2}=e^{\alpha_{2}-\alpha_{1}}$ is a non-zero constant;
 \item [(II)]. $k_{1}\alpha_{1}'=n\alpha_{2}'$, where
 $k_{1}$ is an integer satisfying $1\leq k_{1} \leq n-1$. In this case,
     $f(z)=s_{3}(z)e^{\frac{\alpha_{1}(z)}{n}}$,
 where $s_{3}$ is a rational function satisfying $s_{3}^{n}=p_{1}$;
  \item [(III)]. $k_{2}\alpha_{2}'=n\alpha_{1}'$, where
   $k_{2}$ is an integer satisfying $1\leq k_{2}\leq n-1$.
  In this case, $f(z)=s_{4}(z)e^{\frac{\alpha_{2}(z)}{n}}$, where  $s_{4}$ is a rational function satisfying $s_{4}^{n}=p_{2}$.
\end{itemize}
\end{theorem}

\begin{theorem}\label{colA1z1}
 Let $n\geq 2$ be an integer. Suppose that $P_{*}(f)$ is a differential polynomial in $f(z)$ of degree $n-1$ with rational functions as its coefficients, $\alpha_{1}, \alpha_{2}, p_{1}, p_{2}$ be nonzero constants such that $\alpha_{1}\pm\alpha_{2}\neq 0$. If $f(z)$ is an transcendental meromorphic solution
 of the following nonlinear differential equation
 \begin{equation}\label{aaaa122}
f^{n}(z)+P_{*}(f)=p_{1}e^{\alpha_{1}z}+p_{2}e^{\alpha_{2}z},
\end{equation}
satisfying $N(r,f)=S(r,f)$,
then $\sigma(f)=1$ and
there exist two cases:
\begin{itemize}
  \item [(I)] $N\left(r,\frac{1}{f}\right)=S(r,f)$, then
  one of the following
 relations holds: (a) $k_{1}\alpha_{1}=n\alpha_{2}$ and $f=s_{1}\exp(\alpha_{1}z /n)$; (b) $k_{2}\alpha_{2}=n\alpha_{1}$ and  $f=s_{2}\exp(\alpha_{2}z /n)$, where $k_{1}, k_{2}$ are integers satisfying $1\leq k_{1},  k_{2}\leq n-1$,
     $s_{1}, s_{2}$ are constants with $s_{1}^{n}=p_{1}$ and $s_{2}^{n}=p_{2}$;
  \item [(II)]  $N\left(r,\frac{1}{f}\right)\neq S(r,f)$, then  $T(r,f)\leq N_{1)}\left(r,\frac{1}{f}\right)+\frac{1}{2}T(r,\varphi)
+\frac{1}{2}N\left(r,\frac{1}{\varphi}\right)+S(r,f)$, where
  $\varphi =\alpha_{1}\alpha_{2}f^{2}
-n(\alpha_{1}+\alpha_{2})  ff'
 +n(n-1)(f')^{2}+nff''\not\equiv 0$,
  
  and  \\
 (1). if $\varphi$ is a nonzero constant,  then  one of the following
 relations holds: (a) $(n-1)\alpha_{1}=n\alpha_{2}$ and $f(z)=c_{1}e^{\alpha_{1}z/n}-\frac{1}{\alpha_{1}}\sqrt{\frac{n\varphi}{n-1}}\,  (c_{1}^{n}=p_{1}),$ or
 $f(z)=c_{3}e^{\alpha_{1}z/n}+\frac{1}{\alpha_{1}}\sqrt{\frac{n\varphi}{n-1}}\, (c_{3}^{n}=p_{1})$; (b)
$(n-1)\alpha_{2}=n\alpha_{1}$, and $f(z)=c_{2}e^{\alpha_{2}z/n}-\frac{1}{\alpha_{2}}\sqrt{\frac{n\varphi}{n-1}}, (c_{2}^{n}=p_{2})$ or
$f(z)=c_{4}e^{\alpha_{2}z/n}+\frac{1}{\alpha_{2}}\sqrt{\frac{n\varphi}{n-1}} (c_{4}^{n}=p_{2})$;\\
  (2). if $\varphi$ is a nonconstant meromorphic function, then $T(r,\varphi)\neq S(r,f)$. Particularly, suppose $n=2$ and  $\varphi=P(z)e^{Q(z)}$, where $P$ and  $Q$ are non-vanishing polynomials such that $\deg Q\geq 1$. Then we have
  $\deg Q=1$ and $f^{2}=d_{1}e^{\alpha_{1}z}+d_{2}e^{\alpha_{2}z}-R(z)e^{Q(z)}$, where $d_{1},\, d_{2}$ are
  constants, and $R$ is a non-vanishing polynomial with  $\deg R\leq \deg P+2$.
\end{itemize}

\end{theorem}

The following Examples~\ref{exappp} and \ref{exaqqq} are shown to
  illustrate the cases (II)(1) and (II)(2)
  of Theorem~\ref{colA1z1}.
\begin{exa}\label{exappp}
  $f_{0}=e^{z}-1$ is a solution of the equation
  \begin{eqnarray*}
    f^{2}+2f'+f=e^{2z}+e^{z}.
  \end{eqnarray*}
 Here $\alpha_{1}=2$, $\alpha_{2}=1$, $\alpha_{1}=2\alpha_{2}$ and $\varphi=2$. It implies that case (II)(1)(a) occurs.
\end{exa}

\begin{exa}\label{exaqqq}
  $f_{0}=e^{2z}+e^{z}$ is a solution of
  \begin{eqnarray*}
    f^{2}+\frac{1}{2}f'-\frac{1}{2}f''=e^{4z}+2e^{3z}.
  \end{eqnarray*}
  Here $\alpha_{1}=4$, $\alpha_{2}=3$, $n=2$,  $\varphi=2e^{2z}$, and $f_{0}^{2}=e^{4z}+2e^{3z}+e^{2z}$. It implies that case (II)(2) occurs.
\end{exa}

\section{Preliminary Lemmas}

The following lemma plays an important role in
uniqueness problems of meromorphic functions.

\begin{lemma}[\cite{Yi1}]\label{lemma21}
 Let $f_{j}(z)\, (j=1,\ldots,n)\, (n\geq2)$ be meromorphic
functions, and let $g_{j}(z)\, (j=1,\ldots,n)$ be entire functions satisfying
\begin{itemize}
\item[(i)] $\sum _{j=1}^{n}f_{j}(z)e^{g_{j}(z)}\equiv 0;$
\item[(ii)]when $1\leq j < k \leq n,$ then $g_{j}(z)-g_{k}(z)$ is not a constant;
\item[(iii)] when $1\leq j \leq n, 1\leq h < k \leq n$, then
$$T(r,f_{j})=o \{T(r,e^{g_{h}-g_{k}})\} \quad (r\to \infty, r\not\in E),$$
where $E\subset(1,\infty)$ is of finite linear measure or logarithmic measure.
\end{itemize}
Then, $f_{j}(z)\equiv 0\, (j=1,\ldots ,n)$.
\end{lemma}

\begin{lemma}(the Clunie lemma \cite{Laine}) \label{lemma24}
Let $f$ be a transcendental meromorphic solution of the equation:
 \begin{equation*}
   f^{n}P(z,f)=Q(z,f),
 \end{equation*}
where $P(z,f)$ and $Q(z,f)$ are polynomials in $f$ and
its derivatives with meromorhphic coefficients
$\{a_{\lambda}|\lambda\in I\}$ such that
$m(r,a_{\lambda})=S(r,f)$ for all $\lambda\in I$. If the total
degree of $Q(z,f)$ as a polynomial in $f$ and its derivatives
is at most $n$, then $m(r,P(z,f))=S(r,f)$.
\end{lemma}

\begin{lemma}(the Weierstrass factorization theorem \cite{Conway}) \label{lemma26}
Let $f(z)$ be an entire function, and $a_{1}, a_{2}, \ldots$ denote all nonzero
 zeros of $f(z)$ repeated according to multiplicity, suppose also that $f(z)$ has
 a zero at $z=0$ of multiplicity $m\geq0$. Then there exists an entire function
$g(z)$ and a sequence of nonnegatvie integers $p_{1}, p_{2},\cdots$ such that
\begin{eqnarray*}
  f(z)=E(z)e^{g(z)},
\end{eqnarray*}
where $E(z)=z^{m}\prod_{n=1}^{\infty} E_{p_{n}}\left(\frac{z}{a_{n}}\right)$ is
the canonical product formed by the zeros of $f(z)$, and $E_{n}(z)$ is given by
$$E_{0}(z)=1-z;\; E_{n}(z)=(1-z)exp(z+z^{2}/2+\cdots+z^{n}/n),\, n\geq1.$$
\end{lemma}
A well known fact about Lemma~\ref{lemma26} asserts that $\sigma(E)=\lambda(f)\leq \sigma(f)$, and $\sigma(f)=\sigma(e^{g})$ when $\lambda(f)<\sigma(f)$.

The following lemma, which is a slight generalization of
Tumura--Clunie type theorem,  is referred to
\cite[Corollary]{Hua}, can also see \cite[Theorem 4.3.1]{chen2011}.
\begin{lemma}(\cite{chen2011,Hua})\label{lemma22}
  Suppose that $f(z)$ is meromorphic and not constant in the plane, that
  \begin{equation*}
    g(z)=f(z)^{n}+P_{n-1}(f),
  \end{equation*}
 where $P_{n-1}(f)$ is a differential polynomial  of degree at most $n-1$ in $f$, and that
  \begin{equation*}
    N(r,f)+N\left(r,\frac{1}{g}\right)=S(r,f).
  \end{equation*}
  Then $g(z)=(f+\gamma)^{n}$, where $\gamma$ is meromorphic and $T(r,\gamma)=S(r,f)$.

\end{lemma}

\begin{lemma}\cite{LiP2011} \label{lemma2aa}
Suppose that $f$ is a transcendental meromorphic function,
$a,b,c,d$ are small functions with respect to $f$ and
$acd\not\equiv 0$. If
\begin{eqnarray*}
 af^{2}+bff'+c(f')^{2}=d,
\end{eqnarray*}
then
\begin{eqnarray*}
  c(b^{2}-4ac)\frac{d'}{d}+b(b^{2}-4ac)-c(b^{2}-4ac)'+(b^{2}-4ac)c'=0.
\end{eqnarray*}
\end{lemma}

\begin{lemma}\label{lemma31add41}
Let $\alpha_{1},\, \alpha_{2}$ and $a$ be nonzero constants, and $P_{m}(z)$ be a non-vanishing
polynomial. Then the differential equation
\begin{eqnarray}\label{lemma31add31eq1}
 y''-(\alpha_{1}+\alpha_{2})y'+\alpha_{1}\alpha_{2}y=P_{m}(z)e^{az}
\end{eqnarray}
has a special solution $y^{*}=R(z)e^{az}$, where $R(z)$ is a nonzero
 polynomial with $\deg R \leq  \deg P_{m}+2$.
\end{lemma}

\begin{proof}
Set
\begin{eqnarray}\label{lemma31add31eq0}
  P_{m}(z)=a_{m}z^{m}+a_{m-1}z^{m-1}+\cdots+a_{1}z+a_{0},\quad a_{m}\neq 0.
\end{eqnarray}
We guess
\begin{eqnarray*}\label{lemma31add31eq2}
 y^{*}=R(z)e^{az},\quad \textrm{where}\,
R(z)\, \textrm{ is a polynomial},
\end{eqnarray*}
maybe a special solution of \eqref{lemma31add31eq1}.
By substituting $y^{*}$, $(y^{*})'$, $(y^{*})''$ into equation \eqref{lemma31add31eq1}, and
eliminating $e^{az}$, we get
\begin{eqnarray}\label{lemma31add31eq5}
 R''+(2a-\alpha_{1}-\alpha_{2})R'+\left(a^{2}-
 a(\alpha_{1}+\alpha_{2})+\alpha_{1}\alpha_{2}\right)R=P_{m}(z).
\end{eqnarray}
 We derive the polynomial
solution $R(z)$ by using the method of undetermined coefficients.

Case I. $a\neq \alpha_{1}$ and $a\neq \alpha_{2}$. Then $a^{2}-
 a(\alpha_{1}+\alpha_{2})+\alpha_{1}\alpha_{2}\neq 0$.
 We choose $R(z)$ is a polynomial with degree $m$ as follow,
\begin{eqnarray}\label{lemma31add31eq6}
 R(z)=b_{m}z^{m}+b_{m-1}z^{m-1}+\cdots+b_{1}z+b_{0}.
\end{eqnarray}
By substituting  \eqref{lemma31add31eq0} and \eqref{lemma31add31eq6}
into \eqref{lemma31add31eq5}, comparing the coefficients of the same power of
$z$ at  both sides of equation \eqref{lemma31add31eq5}, we
get the following system of linear equations,
\begin{equation*}
\left\{
\begin{aligned}
&a_{m}= \left(a^{2}-
 a(\alpha_{1}+\alpha_{2})+\alpha_{1}\alpha_{2}\right)b_{m},\\
&a_{m-1}=\left(a^{2}-
 a(\alpha_{1}+\alpha_{2})+\alpha_{1}\alpha_{2}\right)b_{m-1}+
 (2a-\alpha_{1}-\alpha_{2})mb_{m},\\
 &a_{i}=\left(a^{2}-
 a(\alpha_{1}+\alpha_{2})+\alpha_{1}\alpha_{2}\right)b_{i}
 +(2a-\alpha_{1}-\alpha_{2})(i+1)b_{i+1}+(i+2)(i+1)b_{i+2},
 \\
&\quad \quad \quad \quad \quad \quad \quad \quad \quad \quad \quad \quad
  \quad \quad \quad \quad \quad \quad \quad \quad \quad \quad \quad \quad
   i =m-2,\ldots,1,0.\\
\end{aligned}
\right.
\end{equation*}
Since $a^{2}-
 a(\alpha_{1}+\alpha_{2})+\alpha_{1}\alpha_{2}\neq 0$, we can solve $b_{i}\, (i=0,1,\ldots,m)$ by using Cramer's rule to the above system.

Case II. $\alpha_{1}\neq \alpha_{2}$, and either $a=\alpha_{1}$ or $a= \alpha_{2}$.
Then $2a-\alpha_{1}-\alpha_{2}\neq 0$,  and \eqref{lemma31add31eq5} reduces to
 \begin{eqnarray}\label{lemma31add31eq7}
 R''+(2a-\alpha_{1}-\alpha_{2})R'=P_{m}(z).
\end{eqnarray}
 We choose $R(z)$ is a polynomial with degree $m+1$ as follow,
\begin{eqnarray}\label{lemma31add31eq8}
 R(z)=c_{m+1}z^{m+1}+c_{m}z^{m}+\cdots+c_{1}z.
\end{eqnarray}
By substituting  \eqref{lemma31add31eq0} and \eqref{lemma31add31eq8}
into \eqref{lemma31add31eq7}, comparing the coefficients of the same power of
$z$ at  both sides of equation \eqref{lemma31add31eq7}, we
get the following system of linear equations,
\begin{equation*}
\left\{
\begin{aligned}
a_{m}&= (2a-\alpha_{1}-\alpha_{2})(m+1)c_{m+1},\\
a_{i}&=(2a-\alpha_{1}-\alpha_{2})(i+1)c_{i+1}+
 (i+2)(i+1)c_{i+2},\,i =m-1,\ldots,1,0. \\
\end{aligned}
\right.
\end{equation*}
Since $2a-\alpha_{1}-\alpha_{2}\neq 0$, we can solve $c_{i}\, (i=1,\ldots,m+1)$
by using Cramer's rule to the above system.

Case III. $a=\alpha_{1}=\alpha_{2}$. Then $2a-\alpha_{1}-\alpha_{2}=0$, $a^{2}-
 a(\alpha_{1}+\alpha_{2})+\alpha_{1}\alpha_{2}= 0$, and \eqref{lemma31add31eq5} reduces to
\begin{eqnarray}\label{lemma31add31eq9}
 R''=P_{m}(z).
\end{eqnarray}
 We choose $R(z)$ is another polynomial with degree $m+2$ as follow,
\begin{eqnarray}\label{lemma31add31eq10}
 R(z)=d_{m+2}z^{m+2}+d_{m+1}z^{m+1}+\cdots+d_{2}z^{2}.
\end{eqnarray}
By substituting  \eqref{lemma31add31eq0} and \eqref{lemma31add31eq10}
into \eqref{lemma31add31eq9}, comparing the coefficients of the same power of
$z$ at  both sides of equation \eqref{lemma31add31eq9}, we
get the following system of linear equations,
\begin{equation*}
\left\{
\begin{aligned}
a_{m}&= (m+2)(m+1)d_{m+2},\\
a_{m-1}&= (m+1)md_{m+1},\\
&\cdots \\
a_{0}&=2d_{2}. \\
\end{aligned}
\right.
\end{equation*}
 Obviously, we can solve $d_{i}\, (i=2,\ldots,m+2)$ directly from the above system.
\end{proof}

\begin{lemma}\label{lemma23}
Let $n\geq2$ be  integers and $P_{d}(f)$ denote an algebraic differential polynomial in
$f(z)$ of degree $d\leq n-1$ with small functions of $f$ as coefficients. If $p_{1}(z),\, p_{2}(z)$
are small functions of $f$,  $\alpha_{1}(z), \alpha_{2}(z)$ are nonconstant entire functions and if $f$
is a  transcendental meromorphic solution of the equation
\begin{eqnarray}\label{lem2.3zzz1}
  f^{n}+P_{d}(f)=p_{1}e^{\alpha_{1}}+p_{2}e^{\alpha_{2}}
\end{eqnarray}
with $N(r,f)=S(r,f)$, then we have
$T(r,f)=O\left(  T\left(r,p_{1}e^{\alpha_{1}} +p_{2}e^{\alpha_{2}} \right)\right)$,
$T(r, p_{1}e^{\alpha_{1}}+p_{2}e^{\alpha_{2}})=O(T(r,f))$,
and $T(r, f^{n}+P_{d}(f))\neq S(r,f)$.
\end{lemma}

\begin{proof}
By the proof of \cite[Theorem 1.3]{Zhang2018} (or \cite[Lemma 2.4.2.Clunie lemma]{Laine}), we get that
\begin{eqnarray}\label{lem2.2aaa1}
  m\left(r,P_{d}(f) \right)\leq d m(r,f) +S(r,f).
\end{eqnarray}

By combining  \eqref{lem2.3zzz1}, \eqref{lem2.2aaa1} with $N(r,f)=S(r,f)$, we get that
\begin{eqnarray*}
  nT(r,f)&=&T(r,f^{n}) \leq m\left(r,p_{1}e^{\alpha_{1}} +p_{2}e^{\alpha_{2}} \right)+m\left(r,P_{d}(f) \right)+S(r,f)\\
    &\leq&  T\left(r,p_{1}e^{\alpha_{1}} +p_{2}e^{\alpha_{2}} \right)+ dT(r,f)+S(r,f).
\end{eqnarray*}
This gives that
\begin{eqnarray*}
  (n-d)T(r,f)\leq T\left(r,p_{1}e^{\alpha_{1}} +p_{2}e^{\alpha_{2}} \right)+S(r,f),
\end{eqnarray*}
i.e.,
\begin{eqnarray}\label{aaa000}
 T(r,f)=O\left(  T\left(r,p_{1}e^{\alpha_{1}} +p_{2}e^{\alpha_{2}} \right)\right).
\end{eqnarray}

From \eqref{lem2.2aaa1}, $N(r,f)=S(r,f)$ and equation \eqref{lem2.3zzz1}, we
can also get
\begin{eqnarray}\label{aaa001}
  T(r, p_{1}e^{\alpha_{1}}+p_{2}e^{\alpha_{2}})=O(T(r,f)).
\end{eqnarray}

Next, we prove that $T(r, f^{n}+P_{d}(f))$ can not be
a small function of $f$.  Otherwise, we will have $f^{n}+P_{d}(f)=\beta$
with $T(r,\beta)=S(r,f)$. Thus $f^{n}=\beta-P_{d}(f)$.
Since $d\leq n-1$,  from Lemma~\ref{lemma24}, we get $m(r,f)=S(r,f)$. Then  $T(r,f)=S(r,f)$  since $N(r,f)=S(r,f)$, a contradiction.

\end{proof}







\section{Proof of Theorem~\ref{thmABABnew1}.}

Let $f$ be a transcendental meromorphic solution of the equation \eqref{thmABABneweq1} with  $N(r,f)=S(r,f)$.

Since
\begin{eqnarray*}
 N(r,f)+N\left(r,\frac{1}{R(z)e^{\alpha(z)}} \right) = S(r,f),
\end{eqnarray*}
by Lemma~\ref{lemma22} we get
\begin{eqnarray*}
 (f-\gamma)^{n} = R(z)e^{\alpha(z)},\quad T(r,\gamma)=S(r,f).
\end{eqnarray*}
Thus we have
\begin{eqnarray*}
  f=s(z)e^{\alpha(z)/n}+\gamma(z),
\end{eqnarray*}
where $s$ and $\gamma$ are small functions of $f$ with $s^{n}=R$.

\section{Proof of Theorem~\ref{thmA1z1}.}

Let $f$ be a  transcendental meromorphic solution of the equation \eqref{aaaa1} with $\lambda_{f}<\sigma(f)$. Then $f$ is of regular growth, and we have
\begin{eqnarray}\label{thmA1pfeqnnn1}
  N(r,f)=S(r,f),\; \textrm{and} \; N(r,1/f)=S(r,f).
\end{eqnarray}

 By combining with Lemma~\ref{lemma23}, we have
 \begin{equation}\label{thmA1pfeq113}
  T(r, f^{n}+P_{*}(f)\neq S(r,f),
 \end{equation}
and
\begin{eqnarray}\label{thmA1pfeqnnn2bbb}
  \sigma(f)=\sigma(p_{1}e^{\alpha_{1}} +p_{2}e^{\alpha_{2}})=\max\{\deg\alpha_{1}, \deg \alpha_{2}\}.
 \end{eqnarray}

Therefore, by Lemma~\ref{lemma26}   we can factorize $f(z)$ as
\begin{eqnarray}\label{thmA1pfeqnnm1}
 f(z)=\frac{d_{1}(z)}{d_{2}(z)}e^{g(z)}=d(z)e^{g(z)},
\end{eqnarray}
where $g$ is a polynomial with $\deg g=\sigma(f)=\max\{\deg\alpha_{1}, \deg \alpha_{2}\} \geq 1$, $d_{1}$ and $d_{2}$
are the canonical products formed by zeros and poles of $f$ with $\sigma(d_{1})=\lambda(f)<\sigma(f)$ and
 $\sigma(d_{2})=\lambda(1/f)<\sigma(f)$.

Next we assert that $\deg \alpha_{1} =\deg \alpha_{2}$. Otherwise,
we have $\deg \alpha_{1} \neq \deg \alpha_{2}$.

 Suppose that $\deg \alpha_{1}< \deg \alpha_{2}$, then $T(r,e^{\alpha_{1}})=S(r,e^{\alpha_{2}})$. From Lemma~\ref{lemma23}, we get
\begin{equation*}\label{thmA1pfeq111case1eq2}
   (1+o(1))T(r,e^{\alpha_{2}})=T(r, p_{1}e^{\alpha_{1}}+p_{2}e^{\alpha_{2}})\leq K _{1} T(r,f), \quad
   K _{1}>0,
\end{equation*}
which means that a small function of $e^{\alpha_{2}}$
is also a small function of $f$. So we have $T(r,e^{\alpha_{1}})=S(r,f)$.
We rewritten \eqref{aaaa1} as follow:
\begin{equation}\label{thmA1pfeq111case102}
  f^{n}(z)+P_{*}(f)-p_{1}e^{\alpha_{1}}=p_{2}e^{\alpha_{2}}.
\end{equation}
Therefore, by using Theorem~\ref{thmABABnew1}, we get that
$f=s_{0}(z)\exp(\alpha_{2}(z) /n)+t_{0}(z)$, where $s_{0},t_{0}$ are small functions
 of $f$ with $s_{0}^{n}=p_{2}$. If $t_{0}\not\equiv 0$, then
 combining \eqref{thmA1pfeqnnn1} with Nevanlinna's Second Main Theorem,  we have
 \begin{eqnarray*}
    T(r,f)\leq
   N\left(r,\frac{1}{f-t_{0}}\right)+ N\left(r,\frac{1}{f}\right)+N(r,f)+S(r,f)=S(r,f),
 \end{eqnarray*}
a contradiction. So we have $t_{0}\equiv 0$. Moreover, we also have  that $s_{0}$ is a rational function because of the
 fact that $p_{2}$ is a rational function. Substituting $f=s_{0}(z)\exp(\alpha_{2}(z) /n)$ into
 \eqref{thmA1pfeq111case102}, we get that
 \begin{eqnarray*}
   p_{1}e^{\alpha_{1}}=P_{*}(f)=R_{n-1}e^{\frac{n-1}{n}\alpha_{2}}
   +\cdots+R_{1}e^{\frac{1}{n}\alpha_{2}}+R_{0},
 \end{eqnarray*}
where $R_{0},\, R_{1},\ldots, R_{n-1}$ are rational functions. By using Lemma~\ref{lemma21} and $\deg\alpha_{2}>\deg \alpha_{1}>0$, we get that $p_{1}\equiv 0$, a contradiction.

Suppose that $\deg \alpha_{1} > \deg \alpha_{2}$, we can also get a contradiction  as in the case $\deg \alpha_{1}< \deg \alpha_{2}$.

Therefore, $\deg \alpha_{1} = \deg \alpha_{2}$. By combining with \eqref{thmA1pfeqnnn2bbb} and \eqref{thmA1pfeqnnm1},
we have  $\sigma(f)=\deg g=\deg \alpha_{1} = \deg \alpha_{2}$, and
$S(r,f)=S(r,e^{\alpha_{1}})=S(r,e^{\alpha_{2}})$.

{\bf Case 1.}
 $(\alpha_{2}-\alpha_{1})'=0$. Then $\alpha_{2}-\alpha_{1}$ is a constant,
by equation \eqref{aaaa1}, we get
\begin{equation*}
  f^{n}(z)+P_{*}(f)=(p_{1}+p_{2}c_{2})e^{\alpha_{1}}
  =\left(\frac{1}{c_{2}}p_{1}+p_{2}\right)e^{\alpha_{2}},
\end{equation*}
where $c_{2}=e^{\alpha_{2}-\alpha_{1}}$ is a non-zero constant. Obviously, from \eqref{thmA1pfeq113} we have that
 $p_{1}+p_{2}c_{2}\neq 0$ and
 $\frac{1}{c_{2}}p_{1}+p_{2}\neq 0$.
Therefore, by using Theorem~\ref{thmABABnew1}, we get that
$f=s_{1}(z)\exp(\alpha_{1}(z) /n)+t_{1}(z)=s_{2}(z)\exp(\alpha_{2}(z) /n)+t_{2}(z)$, where $s_{1},t_{1}, s_{2},t_{2}$ are small functions
 of $f$ with $s_{1}^{n}=p_{1}+p_{2}c_{2}$ and $s_{2}^{n}=\frac{1}{c_{2}}p_{1}+p_{2}$. Combining \eqref{thmA1pfeqnnn1} with Nevanlinna's Second Main Theorem, we have $t_{1}\equiv 0$ and $t_{2}\equiv 0$.
 From $p_{1},\, p_{2}$ are rational functions,
 we have $s_{1}$ and $s_{2}$ are rational functions. This belongs to Case I in Theorem~\ref{thmA1z1}.

 {\bf Case 2.} $(\alpha_{2}-\alpha_{1})'\neq0$.  By differentiating both sides of \eqref{aaaa1}, we have
\begin{eqnarray}\label{theorem1.1aaa1}
  nf^{n-1}f'+P_{*}'(f)=(p_{1}'+p_{1}\alpha_{1}')e^{\alpha_{1}}
  +(p_{2}'+p_{2}\alpha_{2}')e^{\alpha_{2}}.
\end{eqnarray}
Obviously, we have that $p_{1}'+p_{1}\alpha_{1}'\not\equiv0$ and
$p_{2}'+p_{2}\alpha_{2}'\not\equiv0$. Otherwise, we will
get that $p_{1}=c_{0}e^{-\alpha_{1}}$ and $p_{2}=c_{1}e^{-\alpha_{2}}$, where $c_{0},\, c_{1}\in \mathbb{C}\setminus \{0\}$, which contradict with the fact that $\alpha_{1}, \alpha_{2}$ are nonconstant polynomials, and
 $p_{1}, p_{2}$ are non-vanishing rational functions.

By eliminating $e^{\alpha_{2}}$   from equations \eqref{aaaa1} and
\eqref{theorem1.1aaa1}, we have
\begin{eqnarray}\label{theorem1.1aaa2}
 (p_{2}'+p_{2}\alpha_{2}')f^{n}-np_{2}f^{n-1}f'
 +Q_{1}(f)=A_{1}e^{\alpha_{1}},
\end{eqnarray}

 where
 \begin{eqnarray}\label{theorem1.1nnn8}
  A_{1}=p_{1}\left(p_{2}'+p_{2}\alpha_{2}'\right)-p_{2}\left(p_{1}'+p_{1} \alpha_{1}'\right),
 \end{eqnarray}
 and
  \begin{eqnarray}\label{theorem1.1nnn7}
    Q_{1}(f)=\left(p_{2}'+p_{2}\alpha_{2}'\right)P_{*}-p_{2}P_{*}'.
  \end{eqnarray}


 We assert that
$A_{1}(z)\not\equiv 0$. Otherwise, if $A_{1}(z)\equiv 0$, then we have
\begin{eqnarray*}
  \left(p_{2}'+p_{2}\alpha_{2}'\right)p_{1}=p_{2}\left(p_{1}'+p_{1} \alpha_{1}'\right).
\end{eqnarray*}
Therefore
\begin{eqnarray}\label{thmA1pfeq114}
p_{2}e^{\alpha_{2}}=c_{3}p_{1}e^{\alpha_{1}}, \quad c_{3}\in \mathbb{C}\setminus\{0\}.
\end{eqnarray}
 So we get $\alpha_{2}-\alpha_{1}$
is a constant, a contradiction with the assumption $(\alpha_{2}-\alpha_{1})'\neq0$.
Therefore, $A_{1}(z)\not\equiv 0$.

By differentiating \eqref{theorem1.1aaa2}, we have
\begin{eqnarray}\label{thmA1pfeq111case3eqnnn3}
  &&(p_{2}'+p_{2}\alpha_{2}')'f^{n}+np_{2}\alpha_{2}'f^{n-1}f'
  -np_{2}(n-1)f^{n-2}(f')^{2}-np_{2}f^{n-1}f''
  +Q_{1}'(f)\nonumber \\
  &&\quad =(A_{1}'+A_{1}\alpha_{1}')e^{\alpha_{1}}.
\end{eqnarray}

By eliminating $e^{\alpha_{1}}$ from equations \eqref{theorem1.1aaa2} and \eqref{thmA1pfeq111case3eqnnn3}, we obtain
\begin{eqnarray}\label{theorem1.1aaa5}
&&f^{n-2} \varphi=Q(f),
\end{eqnarray}
where
\begin{eqnarray}\label{rrraaa}
 \varphi&=&\left((A_{1}'+A_{1}\alpha_{1}')(p_{2}'+p_{2}\alpha_{2}')
-A_{1}(p_{2}'+p_{2}\alpha_{2}')'\right)f^{2}
+n(n-1)p_{2}A_{1}(f')^{2}\nonumber\\
&&\; -np_{2}\left(A_{1}'+A_{1}(\alpha_{1}'+\alpha_{2}')  \right)ff'+np_{2}A_{1}ff''.
\end{eqnarray}
and
\begin{eqnarray}\label{theorem1.1aaa8}
  Q(f)=A_{1}Q_{1}'(f)-(A_{1}'+A_{1}\alpha_{1}')Q_{1}(f).
\end{eqnarray}

Next we discuss two cases.

{\bf Subcase 2.1.} $Q(f)\equiv 0$. Then by \eqref{theorem1.1aaa5}, we have
$\varphi\equiv 0$, i.e.,
\begin{eqnarray}\label{theorem1.1aaabbb8}
&&\left((A_{1}'+A_{1}\alpha_{1}')(p_{2}'+p_{2}\alpha_{2}')
-A_{1}(p_{2}'+p_{2}\alpha_{2}')'\right)f^{2}
=np_{2}\left(A_{1}'+A_{1}(\alpha_{1}'+\alpha_{2}')  \right)ff'\nonumber\\
&&\quad -n(n-1)p_{2}A_{1}(f')^{2}-np_{2}A_{1}ff''.
\end{eqnarray}

Next we assert that $f$ has at most finitely many zeros and poles.
Otherwise,   $f$ has infinitely many zeros or poles.

Suppose that $f$ has infinitely many zeros.
let $z_{0}$
be a zero of $f$ with multiplicity $k$ but neither a zero nor a pole of the coefficients
in equation \eqref{theorem1.1aaabbb8}, then $k\geq 2$ and $f(z)=a_{k}(z-z_{0})^{k}+a_{k+1}(z-z_{0})^{k+1}+\cdots \, (a_{k}\neq 0)$ holds
in some small neighborhood of $z_{0}$.

If $(A_{1}'+A_{1}\alpha_{1}')(p_{2}'+p_{2}\alpha_{2}')
-A_{1}(p_{2}'+p_{2}\alpha_{2}')'\equiv 0$, then we have
\begin{eqnarray*}
  \frac{A_{1}'}{A_{1}}+\alpha_{1}'
  =\frac{(p_{2}'+p_{2}\alpha_{2}')'}{p_{2}'+p_{2}\alpha_{2}'}.
\end{eqnarray*}
This gives that
\begin{eqnarray*}
 A_{1}e^{\alpha_{1}}=c_{4}(p_{2}'+p_{2}\alpha_{2}'),\quad c_{4}\in \mathbb{C}\setminus\{0\},
\end{eqnarray*}
which yields a contradiction with $A_{1}(\not\equiv 0),\, p_{2}'+p_{2}\alpha_{2}'(\not\equiv 0)$ are rational functions, and $\alpha_{1}$ is a nonconstant polynomial.  Therefore, $(A_{1}'+A_{1}\alpha_{1}')(p_{2}'+p_{2}\alpha_{2}')
-A_{1}(p_{2}'+p_{2}\alpha_{2}')'\not\equiv 0$.

Obviously, $z_{0}$ is a zero with  multiplicity $2k$ of the left side of \eqref{theorem1.1aaabbb8}. As to the right
side, the coefficient of $(z-z_{0})^{2k-2}$ is
\begin{eqnarray*}
  -nkp_{2}A_{1}((n-1)k+(k-1))a_{k}^{2},
\end{eqnarray*}
which can not equal to zero when $n, k\geq 2$. Therefore, $z_{0}$ is a zero with
multiplicity $2k-2$ of the right side of \eqref{theorem1.1aaabbb8}. This is a contradiction.

 Suppose that $f$ has infinitely many poles. Let $z_{1}$ be a pole of $f$ with multiplicity $m$ but
 neither a zero nor a pole of the coefficients in equation \eqref{theorem1.1aaabbb8}, then
$f(z)=\frac{a_{-m}}{(z-z_{1})^{m}}+\frac{a_{-m+1}}{(z-z_{1})^{m-1}}+\cdots \, (a_{-m}\neq 0)$ holds in some small neighborhood of $z_{1}$.
Obviously, $z_{1}$ is a pole with  multiplicity $2m$ of the left side of \eqref{theorem1.1aaabbb8}. As to the right
side, the coefficient of $(z-z_{0})^{-2(m+1)}$ is
\begin{eqnarray*}
  -nmp_{2}A_{1}((n-1)m+(m+1))a_{-m}^{2},
\end{eqnarray*}
which can not equal to zero when $m\geq 1$ and $n\geq 2$. Therefore, $z_{1}$ is a pole with
multiplicity $2(m+1)$ of the right side of \eqref{theorem1.1aaabbb8}. This is a contradiction.

Therefore, $f$ has at most finitely many zeros and poles. So
\begin{eqnarray}\label{thmA1pfeqnnm2}
 f(z)=d(z)e^{g(z)},
\end{eqnarray}
where $g$ is a polynomial with  $\deg g=\deg \alpha_{1} = \deg \alpha_{2}\geq 1$, and $d$ is a rational function.

By substituting \eqref{thmA1pfeqnnm2} into equation \eqref{aaaa1}, we get that
\begin{eqnarray}\label{thmA1pfeqnnm3}
  d^{n}e^{ng}+\widetilde{R}_{n-1}e^{(n-1)g}+\cdots+\widetilde{R}_{1}e^{g}
  +\widetilde{R}_{0}=p_{1}e^{\alpha_{1}}+p_{2}e^{\alpha_{2}},
\end{eqnarray}
where $\widetilde{R}_{0}, \widetilde{R}_{1}, \ldots, \widetilde{R}_{n-1}$ are rational functions.

If neither $ng(z)-\alpha_{1}(z)$ nor $ng(z)-\alpha_{2}(z)$ are constants, then
by Lemma~\ref{lemma21}, we get that $d(z)\equiv 0$, which yields a contradiction.

If $ng(z)-\alpha_{1}(z)$ is a constant, then $ng(z)-\alpha_{2}(z)$ is not a constant, otherwise we have $\alpha_{2}(z)-\alpha_{1}(z)$ is a constant, which yields a contradiction.  We set $ng(z)-\alpha_{1}(z)=c_{5}$, then
\eqref{thmA1pfeqnnm3} can be reduced to
\begin{eqnarray*}
  (d^{n}-p_{1}e^{-c_{5}})e^{ng}+\widetilde{R}_{n-1}e^{(n-1)g}+\cdots+\widetilde{R}_{1}e^{g}
  +\widetilde{R}_{0}-p_{2}e^{\alpha_{2}}=0.
\end{eqnarray*}
By Lemma~\ref{lemma21}, there must exists some integer $k_{1}\,  (1\leq k_{1}\leq n-1)$ such that
\begin{eqnarray*}
 k_{1}g'=\alpha_{2}' \;\; \textrm{and} \;\; d^{n}-p_{1}e^{-c_{5}}=0.
\end{eqnarray*}
Therefore, by combining with \eqref{thmA1pfeqnnm2} we have
\begin{eqnarray*}
 f(z)=s_{3}(z)e^{\frac{\alpha_{1}(z)}{n}},
\end{eqnarray*}
where $s_{3}^{n}=p_{1}$, and $k_{1}\alpha_{1}'=n\alpha_{2}'$.

If $ng(z)-\alpha_{2}(z)$ is a constant, then $ng(z)-\alpha_{1}(z)$ is not a constant, following the similar reason, we have
\begin{eqnarray*}
 f(z)=s_{4}(z)e^{\frac{\alpha_{2}(z)}{n}},
\end{eqnarray*}
where $s_{4}^{n}=p_{2}$, and $k_{2}\alpha_{2}'=n\alpha_{1}'\, (1\leq k_{2}\leq n-1)$.

{\bf Subcase 2.2.} $Q(f)\not\equiv 0$. By combining Logarithmic Derivative Lemma
with \eqref{rrraaa}, we get
\begin{eqnarray}\label{thmA1pfeqnnm4add1}
  m\left(r, \frac{\varphi}{f^{2}}\right)=S(r,f).
\end{eqnarray}
We rewritten \eqref{theorem1.1aaa5} as follow,
\begin{eqnarray}\label{thmA1pfeqnnm4add2}
  f^{n-1}\frac{\varphi}{f} =Q(f).
\end{eqnarray}
From \eqref{rrraaa}, we have
\begin{eqnarray}\label{thmA1pfeqnnm4add3}
  \frac{\varphi}{f}&=&\left((A_{1}'+A_{1}\alpha_{1}')(p_{2}'+p_{2}\alpha_{2}')
-A_{1}(p_{2}'+p_{2}\alpha_{2}')'\right)f
+n(n-1)p_{2}A_{1}\frac{f'}{f}\cdot f'\nonumber\\
&&\; -np_{2}\left(A_{1}'+A_{1}(\alpha_{1}'+\alpha_{2}')  \right)f'+np_{2}A_{1}f'',
\end{eqnarray}
is a polynomial in $f,\, f'$ and $f''$ with meromorhphic coefficients
 such that
\begin{eqnarray*}
  m\left(r, (A_{1}'+A_{1}\alpha_{1}')(p_{2}'+p_{2}\alpha_{2}')
-A_{1}(p_{2}'+p_{2}\alpha_{2}')'\right)=S(r,f),\; m(r,p_{2}A_{1})=S(r,f),
\end{eqnarray*}
\begin{eqnarray*}
 m\left(r,p_{2}A_{1}\frac{f'}{f}\right)=S(r,f),\; \textrm{and}\; m\left(r,p_{2}\left(A_{1}'+A_{1}(\alpha_{1}'+\alpha_{2}') \right)\right)=S(r,f).
\end{eqnarray*}

By
combining with \eqref{thmA1pfeqnnm4add2}, \eqref{thmA1pfeqnnm4add3}, \eqref{theorem1.1aaa8}, and Lemma~\ref{lemma24}, we have that
\begin{eqnarray}\label{thmA1pfeqnnm4}
  m\left(r, \frac{\varphi}{f}\right)=S(r,f).
\end{eqnarray}

From\eqref{thmA1pfeqnnn1}, \eqref{rrraaa}, \eqref{thmA1pfeqnnm4add1} and \eqref{thmA1pfeqnnm4}, we get that
\begin{eqnarray*}
  2T(r,f)+S(r,f)&=&T\left(r,\frac{1}{f^{2}}\right)= m\left(r,\frac{1}{f^{2}}\right)+S(r,f)\\
  &\leq& m\left(r,\frac{\varphi}{f^{2}}\right)+m\left(r,\frac{1}{\varphi}\right)+S(r,f)\\
  &\leq& T(r,\varphi)+S(r,f)\\
  &\leq& m\left(r,\frac{\varphi}{f} \right)+m(r,f)+S(r,f)\\
   &\leq& T(r,f)+S(r,f),
\end{eqnarray*}
which yields a contradiction.




\section{Proof of Theorem~\ref{colA1z1}.}
Let $f$ be a  transcendental meromorphic solution of the equation \eqref{aaaa122} with $N(r,f)=S(r,f)$. By Lemma~\ref{lemma23}, we have that $f$ is of finite order
and
\begin{eqnarray}\label{vvvaaa}
 \sigma(f)=\sigma(p_{1}e^{\alpha_{1}z}+p_{2}e^{\alpha_{2}z})=1.
\end{eqnarray}

If $N(r,1/f)=S(r,f)$, by the proof of Theorem~\ref{thmA1z1}, we can get the conclusion.

Next, we consider the case when $N(r,1/f)\neq S(r,f)$. By differentiating \eqref{aaaa122}, we get
\begin{eqnarray}\label{col1.1bbb1}
  nf^{n-1}f'+P_{*}'(f)=p_{1}\alpha_{1}e^{\alpha_{1}z}+p_{2}\alpha_{2}e^{\alpha_{2}z}
\end{eqnarray}

By eliminating  $e^{\alpha_{2}z}$ from \eqref{aaaa122} and \eqref{col1.1bbb1}, we have
\begin{eqnarray}\label{col1.1bbb2}
 \alpha_{2}f^{n}+\alpha_{2}P_{*}(f)-nf^{n-1}f'-P_{*}'(f)=p_{1}(\alpha_{2}-\alpha_{1})e^{\alpha_{1}z}.
\end{eqnarray}
Differentiating \eqref{col1.1bbb2} yields
\begin{eqnarray}\label{col1.1bbb3}
  n\alpha_{2}f^{n-1}f'+\alpha_{2}P_{*}'
  -n(n-1)f^{n-2}(f')^{2}-nf^{n-1}f''-P_{*}''
  =p_{1}\alpha_{1}(\alpha_{2}-\alpha_{1})e^{\alpha_{1}z}.\nonumber \\
\end{eqnarray}
It follows from \eqref{col1.1bbb2} and \eqref{col1.1bbb3} that
\begin{eqnarray}\label{col1.1aaa5}
f^{n-2} \varphi =
 -P_{*}''+(\alpha_{1}+\alpha_{2})P_{*}'-\alpha_{1}\alpha_{2}P_{*},
\end{eqnarray}
where
\begin{eqnarray}\label{col1.1aaa11}
  \varphi(z)=\alpha_{1}\alpha_{2}f^{2}
-n(\alpha_{1}+\alpha_{2})  ff'
 +n(n-1)(f')^{2}+nff''.
\end{eqnarray}

Next we assert that $\varphi(z)\not\equiv 0$. Otherwise, we have
\begin{eqnarray}\label{col1.1aaa9}
\alpha_{1} \alpha_{2}f^{2}
-n(\alpha_{1}+\alpha_{2})  ff'
 +n(n-1)(f')^{2}+nff''=0.
\end{eqnarray}
Since $N(r,1/f)\neq S(r,f)$, let $z_{0}$
be a zero of $f$ with multiplicity $k$. By \eqref{col1.1aaa9} we have
$k\geq 2$ and $f(z)=a_{k}(z-z_{0})^{k}+a_{k+1}(z-z_{0})^{k+1}+\cdots \, (a_{k}\neq 0)$ holds in some small neighborhood of $z_{0}$. We rewrite \eqref{col1.1aaa9}
as follow,
\begin{eqnarray}\label{col1.1aaa10}
\alpha_{1}\alpha_{2}f^{2}=
n(\alpha_{1}+\alpha_{2})ff'
 -n(n-1)(f')^{2}-nff''.
\end{eqnarray}
Obviously, $z_{0}$ is a zero with  multiplicity $2k$ of the left side of \eqref{col1.1aaa10}. As to the right
side, the coefficient of $(z-z_{0})^{2k-2}$ is
\begin{eqnarray*}
  -nk((n-1)k+(k-1))a_{k}^{2},
\end{eqnarray*}
which can not equal to zero when $n, k\geq 2$. Therefore, $z_{0}$ is a zero with
multiplicity $2k-2$ of the right side of \eqref{col1.1aaa10}. This is a contradiction. Therefore,  $\varphi(z)\not\equiv 0$.

From \eqref{col1.1aaa5} and \eqref{col1.1aaa11}, by using Lemma~\ref{lemma24} and
Logarithmic Derivative Lemma, we have
\begin{eqnarray}\label{col1.1aaa15}
m\left(r, \frac{\varphi}{f}\right)=S(r,f), \; \textrm{and}\;
m\left(r, \frac{\varphi}{f^{2}}\right)=S(r,f).
\end{eqnarray}
From \eqref{col1.1aaa15}, we have
\begin{eqnarray}\label{col1.1mmmaaa}
 2m\left(r,\frac{1}{f}\right)=m\left(r,\frac{1}{f^{2}}\right)
 \leq m\left(r,\frac{\varphi}{f^{2}}\right)+m\left(r,\frac{1}{\varphi}\right)
 \leq m\left(r,\frac{1}{\varphi}\right)+S(r,f).
\end{eqnarray}
By \eqref{col1.1aaa11}, we have
\begin{eqnarray}\label{col1.1mmmbbb}
 N\left(r,\frac{1}{f}\right)=N_{1)}\left(r,\frac{1}{f}\right)
 +N_{(2}\left(r,\frac{1}{f}\right)\leq N_{1)}\left(r,\frac{1}{f}\right)
 + N\left(r,\frac{1}{\varphi}\right)+S(r,f).
\end{eqnarray}
Combining with \eqref{col1.1mmmaaa} and \eqref{col1.1mmmbbb}, we have
\begin{eqnarray*}
T(r,f)\leq N_{1)}\left(r,\frac{1}{f}\right)+\frac{1}{2}T(r,\varphi)
+\frac{1}{2}N\left(r,\frac{1}{\varphi}\right)+S(r,f).
\end{eqnarray*}

{\bf Case 1.} $\varphi(z)$ is a nonzero constant.  Since $N(r,1/f)\neq S(r,f)$, let $z_{1}$ be a zero of $f$ with multiplicity $m$. By \eqref{col1.1aaa11} we have $n(n-1)(f')^{2}(z_{1})=\varphi\neq 0.$ Thus, $m=1$, i.e.,  $z_{1}$ is a simple zero of $f$.

{\bf Subcase 1.1.} $z_{1}$ is a
zero of $f'(z)-\sqrt{\varphi/n(n-1)}$. Then we set
\begin{eqnarray}\label{col1.1aaa13}
  h(z)=\frac{f'(z)-\sqrt{\frac{\varphi}{n(n-1)}}}{f(z)}.
\end{eqnarray}
Obviously, meromorphic function $h(z)\not\equiv 0$.   Otherwise, $f$ will be a polynomial, a contradiction. By \eqref{col1.1aaa15},  Logarithmic Derivative Lemma and $N(r,f)=S(r,f)$, we get $T(r,h)=m(r,h)+N(r,h)=S(r,f)$.  Therefore, $h(z)$ is a small function of $f$. We rewrite \eqref{col1.1aaa13} as follow,
\begin{eqnarray}\label{col1.1aaa16}
  f'=hf+\sqrt{\frac{\varphi}{n(n-1)}},
\end{eqnarray}
then,
\begin{eqnarray}\label{col1.1aaa17}
 f''=h'f+hf'=(h^{2}+h')f+h\sqrt{\frac{\varphi}{n(n-1)}}.
\end{eqnarray}
Substituting \eqref{col1.1aaa16} and \eqref{col1.1aaa17} into \eqref{col1.1aaa11}, we get that
\begin{eqnarray}\label{yyyaaa}
  A_{1}f+ A_{2}=0,
\end{eqnarray}
where
\begin{eqnarray*}
 A_{1}= \alpha_{1}\alpha_{2}-n(\alpha_{1}+\alpha_{2})h+n^{2}h^{2}+nh'
\end{eqnarray*}
and
\begin{eqnarray*}
  A_{2}=\left((2n-1)h-(\alpha_{1}+\alpha_{2})\right)\sqrt{n\varphi/(n-1)} .
\end{eqnarray*}

Suppose that $\alpha_{1}+\alpha_{2}-(2n-1)h\not\equiv 0$, then by  \eqref{yyyaaa} and $T(r,h)=S(r,f)$,
we have
\begin{eqnarray*}
  N\left(r,\frac{1}{f}\right) \leq N \left(r,\frac{1}{ A_{2}}  \right)
  +N(r,A_{1})=S(r,f),
\end{eqnarray*}
 a contradiction with the assumption that $N(r,1/f)\neq S(r,f)$. Therefore,
 combining with \eqref{yyyaaa}   we have
\begin{equation*}
\left\{
\begin{aligned}
\alpha_{1}\alpha_{2}-n(\alpha_{1}+\alpha_{2})h
  +n^{2}h^{2}+nh'&\equiv 0,\\
\alpha_{1}+\alpha_{2}-(2n-1)h& \equiv 0.\\
\end{aligned}
\right.
\end{equation*}
 Thus
\begin{eqnarray*}\label{col1.1aaa22}
  (n-1)\alpha_{1}=n\alpha_{2}\quad \textrm{or} \quad (n-1)\alpha_{2}=n\alpha_{1}.
\end{eqnarray*}

If $(n-1)\alpha_{1}=n\alpha_{2}$, then $h=\frac{\alpha_{2}}{n-1}=\frac{\alpha_{1}}{n}$, and $f'-\frac{\alpha_{1}}{n}f=\sqrt{\frac{\varphi}{n(n-1)}}.$ Thus the general solutions can be represented in the form $ f(z)=c_{1}e^{\frac{\alpha_{1}}{n}z}
-\frac{1}{\alpha_{1}}\sqrt{\frac{n\varphi}{n-1}},$
where $c_{1}$ is a constant. By substituting it into equation \eqref{aaaa122}, we
get  $c_{1}^{n}=p_{1}$.

If $(n-1)\alpha_{2}=n\alpha_{1}$, then $h=\frac{\alpha_{1}}{n-1}=\frac{\alpha_{2}}{n}$, and $f'-\frac{\alpha_{2}}{n}f=\sqrt{\frac{\varphi}{n(n-1)}}$. Thus the  solution can be represented in the form $ f(z)=c_{2}e^{\frac{\alpha_{2}}{n}z}
-\frac{1}{\alpha_{2}}\sqrt{\frac{n\varphi}{n-1}}$,
where $c_{2}$ is a constant satisfying $c_{2}^{n}=p_{2}$.

{\bf Subcase 1.2.} $z_{1}$ is a zero of  $f'(z)+\sqrt{\varphi/n(n-1)}$. By using the similar arguments
as above, we can get the conclusions that $(n-1)\alpha_{1}=n\alpha_{2}$ and  $f(z)=c_{3}e^{\frac{\alpha_{1}}{n}z}+\frac{1}{\alpha_{1}}\sqrt{\frac{n\varphi}{n-1}}$,
or $(n-1)\alpha_{2}=n\alpha_{1}$ and
$f(z)=c_{4}e^{\frac{\alpha_{2}}{n}z}+\frac{1}{\alpha_{2}}\sqrt{\frac{n\varphi}{n-1}}$,
where $c_{3},\, c_{4}$ are constants satisfying $c_{3}^{n}=p_{1}$ and $c_{4}^{n}=p_{2}$.

{\bf Case 2.}  $\varphi(z)$ is a nonconstant small function of $f$. Differentiating \eqref{col1.1aaa11} gives
\begin{eqnarray}\label{col1.1aaa23}
 \varphi'=2\alpha_{1}\alpha_{2}ff'-n(\alpha_{1}+\alpha_{2})(f')^{2}
 -n(\alpha_{1}+\alpha_{2})ff''+n(2n-1)f'f''+nff'''.\nonumber\\
\end{eqnarray}
It follows from \eqref{col1.1aaa11} and \eqref{col1.1aaa23} that
\begin{eqnarray}\label{col1.1nnn111}
&& \alpha_{1}\alpha_{2}\varphi'f^{2}
-\left[n(\alpha_{1}+\alpha_{2})\varphi'+2\alpha_{1}\alpha_{2}\varphi\right] ff'
 +n\left[(n-1)\varphi'+(\alpha_{1}+\alpha_{2})\varphi\right](f')^{2}\nonumber\\
 &&
 +n\left[(\alpha_{1}+\alpha_{2})\varphi +\varphi'\right]ff''
 -n(2n-1)\varphi f'f''-n\varphi ff'''=0.
\end{eqnarray}
Since $N(r,1/f)\neq S(r,f)$ and $T(r,\varphi)=S(r,f)$, let $z_{2}$ be a zero of $f$, which is neither
 a zero of $\varphi$ nor a pole  of the coefficients in \eqref{col1.1nnn111}, with multiplicity $l$, then by \eqref{col1.1aaa11}
 we have $l=1$, i.e.,  $z_{2}$ is a simple zero of $f$.  And it follows
from \eqref{col1.1nnn111} that $z_{2}$ is also a zero of
$\left[(n-1)\varphi'+(\alpha_{1}+\alpha_{2})\varphi\right]f'-(2n-1)\varphi f''$.

We set
\begin{eqnarray}\label{col1.1aaa24}
g=\frac{(2n-1)\varphi f''-\left[(n-1)\varphi'+(\alpha_{1}+\alpha_{2})\varphi\right]f'}{f},
\end{eqnarray}
then by combining with Logarithmic Derivative Lemma, $N(r,f)=S(r,f)$,
 and $T(r,\varphi)=S(r,f)$, we have
\begin{eqnarray*}
  T(r,g)=O\left(m(r,\varphi)+
  N\left(r,\frac{1}{\varphi}\right)+N(r,\varphi)+N(r,f)\right)+S(r,f)=S(r,f),
\end{eqnarray*}
i.e., $g$ is a small function of $f$. We rewrite \eqref{col1.1aaa24} as follow,
\begin{eqnarray}\label{col1.1aaa25}
  f''=t_{1}f'+\frac{g}{(2n-1)\varphi}f,\quad \textrm{where}\; t_{1}=\frac{1}{2n-1}\left((n-1)\frac{\varphi'}{\varphi}+\alpha_{1}+\alpha_{2} \right).
\end{eqnarray}

Differentiating \eqref{col1.1aaa25} gives that
\begin{eqnarray}\label{col1.1nnn112}
  f'''= \left(t_{1}^{2}+t_{1}'+\frac{g}{(2n-1)\varphi}\right)f'
  +\frac{1}{2n-1}\left(t_{1}\frac{g}{\varphi}+\left(\frac{g}{\varphi} \right)'\right)f.
  \end{eqnarray}

  By substituting \eqref{col1.1aaa25} and  \eqref{col1.1nnn112} into \eqref{col1.1nnn111}, combining with $\varphi\not\equiv 0$,  we get
\begin{eqnarray}\label{tttaaa}
 B_{1}f=B_{2}f',
\end{eqnarray}
where
\begin{eqnarray*}
 B_{1}=\alpha_{1}\alpha_{2}\frac{\varphi'}{\varphi}+n\left(\alpha_{1}+\alpha_{2} +\frac{\varphi'}{\varphi}\right)\frac{g}{(2n-1)\varphi} - \frac{n }{2n-1}\left(\left(\frac{g}{\varphi}\right)'+
  t_{1}\frac{g}{\varphi} \right),
\end{eqnarray*}
and
\begin{eqnarray*}
 B_{2}=n(\alpha_{1}+\alpha_{2})\left(\frac{\varphi'}{\varphi}-t_{1}\right)+
 2\alpha_{1}\alpha_{2}
-n \frac{\varphi'}{\varphi}t_{1}+\frac{ng}{\varphi}
+n\left(t_{1}'+\frac{g}{(2n-1)\varphi}+t_{1}^{2}\right).
\end{eqnarray*}

If $B_{2}\not\equiv 0$, then from \eqref{tttaaa} and $f$ is transcendental, we have $B_{1}\not\equiv 0$. Since $N(r,1/f)\neq S(r,f)$, $T(r,\varphi)=S(r,f)$, and $T(r,g)=S(r,f)$,
let $z_{3}$
be a zero of $f$ with multiplicity $q$, which is neither a zero nor a pole of  $B_{1}$ and $B_{2}$. Then  $z_{3}$ is a zero with
multiplicity $q$ of the left side of  \eqref{tttaaa}, but a zero with
multiplicity $q-1$ of the right side, which yields a contradiction.
Therefore, we have $B_{2}\equiv 0$ and $B_{1}\equiv 0$, i.e.,
\begin{eqnarray}\label{col1.1nnn113}
\left(\frac{g}{\varphi}\right)'= \left(\frac{2(n-1)}{2n-1} (\alpha_{1}+\alpha_{2}) +\frac{n}{2n-1} \gamma\right)\frac{g}{\varphi}
&&+\frac{2n-1}{n}\alpha_{1}\alpha_{2}\gamma,
\end{eqnarray}
and
\begin{eqnarray}\label{col1.1nnn114}
  -\frac{2n}{2n-1}\frac{g}{\varphi}&=&(\alpha_{1}+\alpha_{2})\gamma
 +\frac{2}{n}\alpha_{1}\alpha_{2}-\frac{1}{2n-1}(\alpha_{1}+\alpha_{2} +\gamma)(\alpha_{1}+\alpha_{2}+(n-1)\gamma)\nonumber\\
 &&+
\frac{1}{(2n-1)^{2}}(\alpha_{1}+\alpha_{2} +(n-1)\gamma)^{2}+\frac{n-1}{2n-1}\gamma',
\end{eqnarray}
where $\gamma=\frac{\varphi'}{\varphi}$.

Substituting \eqref{col1.1aaa25} into
\eqref{col1.1aaa11},
\begin{eqnarray*}
  \varphi(z) = af^{2}+bff'+n(n-1)(f')^{2}.
\end{eqnarray*}
where
\begin{eqnarray*}
 a=\alpha_{1}\alpha_{2}+\frac{n}{2n-1}\frac{g}{\varphi},\quad \textrm{and}
 \quad b=\frac{n(n-1)}{2n-1} \left(\gamma  -2(\alpha_{1}+\alpha_{2})\right).
\end{eqnarray*}

If $a\not\equiv 0$, then by Lemma~\ref{lemma2aa}, we have
\begin{eqnarray}\label{col1.1mmm111}
  &&n(n-1)(b^{2}-4an(n-1))\frac{\varphi'}{\varphi}+b(b^{2}-4an(n-1))\nonumber\\
  &&\quad -n(n-1)(b^{2}-4an(n-1))'=0.
\end{eqnarray}

Suppose that $b^{2}-4an(n-1)\not\equiv 0$. It follows from \eqref{col1.1mmm111} that
\begin{eqnarray}\label{col1.1mmm112}
  2n\frac{\varphi'}{\varphi}=(2n-1)\frac{(b^{2}-4an(n-1))'}{b^{2}-4an(n-1)} +2(\alpha_{1}+\alpha_{2}).
\end{eqnarray}
By integration, we see that there exists a $c_{5}\in \mathbb{C}\setminus\{0\}$ such that
\begin{eqnarray*}
  e^{2(\alpha_{1}+\alpha_{2})z}=c_{5}\varphi^{2n}(b^{2}-4an(n-1))^{-(2n-1)},
\end{eqnarray*}
which implies $e^{2(\alpha_{1}+\alpha_{2})z}\in S(r,f)$, then $\alpha_{2}=-\alpha_{1}$, a contradiction.

Suppose that $b^{2}-4an(n-1)\equiv 0$. Then we have
\begin{eqnarray}\label{col1.1mmm113}
\frac{n(n-1)}{(2n-1)^{2}} \left(\gamma  -2(\alpha_{1}+\alpha_{2})\right)^{2}=4
\left(\alpha_{1}\alpha_{2}+\frac{n}{2n-1}\frac{g}{\varphi}\right).
\end{eqnarray}
Differentiating \eqref{col1.1mmm113} yields
\begin{eqnarray}\label{col1.1mmm114}
&&
\frac{n-1}{2n-1} \left(\gamma-2(\alpha_{1}+\alpha_{2})\right)\gamma'
=2\left(\frac{g}{\varphi} \right)'.
\end{eqnarray}
Differentiating \eqref{col1.1nnn114} yields
\begin{eqnarray}\label{col1.1mmm1155}
  2\left(\frac{g}{\varphi}\right)'=\frac{2(n-1)}{2n-1}\gamma\gamma'-\frac{(2n+1)(n-1)}{(2n-1)n}(\alpha_{1}+\alpha_{2})\gamma'
  -\frac{n-1}{n}\gamma''.
\end{eqnarray}

Combining with \eqref{col1.1mmm114} and \eqref{col1.1mmm1155}, we obtain that
\begin{eqnarray}\label{col1.1mmm115}
n\gamma\gamma'=(\alpha_{1}+\alpha_{2})\gamma'+(2n-1)\gamma''.
\end{eqnarray}

We assert that $\gamma'\not\equiv 0$. Otherwise, by $\gamma'\equiv 0$ and $\varphi$ is
nonconstant we have
\begin{eqnarray*}
  \frac{\varphi'}{\varphi} =c_{6}, \; c_{6}\in \mathbb{C}\setminus\{0\}.
\end{eqnarray*}
Then
\begin{eqnarray*}
  \varphi=c_{7}e^{c_{6}z},  \; c_{7}\in \mathbb{C}\setminus\{0\},
\end{eqnarray*}
 which contradicts with the assumption that $\varphi$ is a nonconstant small function of $f$.

Therefore, \eqref{col1.1mmm115} gives that
\begin{eqnarray}\label{col1.1mmm116}
\alpha_{1}+\alpha_{2}=n\gamma-(2n-1)\frac{\gamma''}{\gamma'}.
\end{eqnarray}

Thus
\begin{eqnarray*}
c_{8}e^{(\alpha_{1}+\alpha_{2})z} =\varphi^{n}\left(\left( \frac{\varphi'}{\varphi} \right)'\right)^{-(2n-1)}, \; c_{8}\in \mathbb{C}\setminus \{0\},
\end{eqnarray*}
which implies that $e^{(\alpha_{1}+\alpha_{2})z}\in S(r,f)$, then
$\alpha_{2}=-\alpha_{1}$, a contradiction.

If $a\equiv 0$, that is
$\frac{g}{\varphi}=-\frac{2n-1}{n}\alpha_{1}\alpha_{2}$.
By substituting it
into \eqref{col1.1nnn113}, we get
\begin{eqnarray*}
  \frac{\varphi'}{\varphi}=2\left(\alpha_{1}+\alpha_{2} \right).
\end{eqnarray*}

So we have
\begin{eqnarray*}
  \varphi=c_{9}e^{2(\alpha_{1}+\alpha_{2} )z}, \; c_{9}\in \mathbb{C}\setminus \{0\},
\end{eqnarray*}
which implies that $e^{2(\alpha_{1}+\alpha_{2} )z}\in S(r,f)$, then
$\alpha_{2}=-\alpha_{1}$, a contradiction.

{\bf Case 3.} $n=2$ and $\varphi(z)=P(z)e^{Q(z)}$, where $P,\, Q$ are nonvanishing polynomials and
$Q$ is non-constant. By  \eqref{vvvaaa} and \eqref{col1.1aaa11},
 we get $\sigma(\varphi)\leq \sigma(f)=1$, combining with $\deg Q\geq 1$, we have $\deg Q=\sigma(\varphi)=1$. Let $Q(z)=az+b$, where $a(\neq 0),b$ are constants,
  then
$\varphi=e^{b}P e^{az}$. By \eqref{col1.1aaa5} we get that
\begin{eqnarray}\label{qqq111}
  P_{*}''-(\alpha_{1}+\alpha_{2})P_{*}'+\alpha_{1}\alpha_{2}P_{*}=-e^{b}P(z)e^{az}.
\end{eqnarray}

From Lemma~\ref{lemma31add41} and the theory of ordinary differential equations, the general solutions of equation \eqref{qqq111} can be represented in the form
\begin{eqnarray}\label{qqq222}
  P_{*}=c_{10}e^{\alpha_{1}z}+c_{11}e^{\alpha_{2}z}+R(z)e^{Q(z)},
\end{eqnarray}
where $c_{10}, c_{11}$ are constants, and $R$ is a polynomial with $\deg R \leq \deg P+2$.

By combining with \eqref{aaaa122}, we get
\begin{eqnarray*}
  f^{2}=d_{1}e^{\alpha_{1}z}+ d_{2}e^{\alpha_{2}z}-R(z)e^{Q(z)},
\end{eqnarray*}
where $d_{1}=p_{1}-c_{10}$, and $d_{2}=p_{2}-c_{11}$.

\section*{Acknowledgments}

This work was supported by  NNSF of China (No.11801215 \& No.11626112 \& No. 11371225), and the NSF of Shandong Province, P. R. China (No.ZR2016AQ20 \& No. ZR2018MA021).


\section*{References}


\def\cprime{$'$}

\end{document}